\documentclass[a4paper, 11pt]{article}

\usepackage{amsfonts,enumerate,amsmath} 
\usepackage{enumerate}
\usepackage{graphicx,psfrag}
\newtheorem{proposition}{Proposition}[section]
\newtheorem{theorem}[proposition]{Theorem}

\newtheorem{corollary}[proposition]{Corollary}

\newtheorem{remark}[proposition]{Remark}

\newenvironment{proof}{\smallskip\noindent\emph{Proof.}\hspace{1pt}}%
{\hspace{-5pt}{\nobreak\quad\nobreak\hfill\nobreak$\square$\vspace{8pt}%
\par}\smallskip\goodbreak}

\newcommand{\D}{\mathcal{D}}

\newcommand{\reali}{{\mathbb{R}}}

\newcommand{\naturali}{{\mathbb{N}}}

\renewcommand{\epsilon}{\varepsilon}
\renewcommand{\phi}{\varphi}
\renewcommand{\theta}{\vartheta}

\title{A simple counterexample related to the Lie--Trotter product formula}

\author{Canzi Claudia \\ \footnotesize Milano--Bicocca University
  \and Graziano Guerra\\ \footnotesize Milano--Bicocca University}

\begin{document}

\maketitle

\begin{abstract}
  In this note a very simple example is given which shows
  that if the sum of two semigroup generators is itself a generator,
  the generated semigroup in general can not be rapresented by the
  Lie-Trotter product formula. 
\end{abstract}

\section{Introduction}
In 1959, H.F. Trotter \cite{Trotter} extended the Lie product formula for matrices
\begin{displaymath}
  e^{A+B}=\lim_{n\to\infty}\left(e^{\frac{1}{n}A}e^{\frac{1}{n}B}\right)^{n}
\end{displaymath}
to unbounded operator in Banach spaces.
The original result can be generalized and one can prove the following
theorem (see \cite[Chapter III, Corollary 5.8]{EngelNagel})
\begin{theorem}
  Let $\left(S\left(t\right)\right)_{t\ge 0}$ and
  $\left(T\left(t\right)\right)_{t\ge 0}$
  be strongly continuous semigroups on a Banach space $X$ satisfying
  the stability condition
  \begin{equation}
      \label{eq:stability}
    \left\|
      \left[S\left(\frac{t}{n}\right)T\left(\frac{t}{n}\right)\right]^{n}\right\|
    \le M e^{wt}, \quad\text{ for all }t\ge 0,\;
    n\in\naturali\setminus \left\{0\right\}
  \end{equation}
  and for constants $M\ge 1$, $w\in\reali$. Consider the ``sum'' $A+B$ on
  $\D:=\D(A)\cap\D(B)$ of the generators
  $\left(A,\D\left(A\right)\right)$ of
  $\left(S\left(t\right)\right)_{t\ge 0}$ and
  $\left(B,\D\left(B\right)\right)$ of
  $\left(T\left(t\right)\right)_{t\ge 0}$, and assume that $\D$ and
  $\left(\lambda_{0} -A -B\right)\D$ are dense in $X$ for some
  $\lambda_{0}> w$. Then $C:=\overline{A+B}$ generates a strongly
  continuous semigroup $\left(U(t)\right)_{t\ge 0}$ given by the
  \emph{Trotter product formula}:
  \begin{equation}
    \label{eq:trotter}
    U(t)x=\lim_{t\to\infty}\left[S\left(\frac{t}{n}\right)T\left(\frac{t}{n}\right)\right]^{n}x,
    \quad\text{ for all }x\in X,
  \end{equation}
  with uniform convergence for $t$ in compact intervals.
\end{theorem}
From this theorem, the following corollary can be deduced
\begin{corollary}
  \label{cor:12}
  Let $\left(S\left(t\right)\right)_{t\ge 0}$ and
  $\left(T\left(t\right)\right)_{t\ge 0}$
  be strongly continuous semigroups on a Banach space $X$,
  with generators 
  $\left(A,\D\left(A\right)\right)$ and
  $\left(B,\D\left(B\right)\right)$ respectively, such that:
  \begin{enumerate}[a)]
  \item 
    \label{item:stability}
    the stability condition \eqref{eq:stability} holds;
  \item
    \label{item:closure}
    the closure of the sum of the generators
    $C=\overline{A+B}$ generates a strongly continuous
    semigroups $\left(U\left(t\right)\right)_{t\ge 0}$.
  \end{enumerate}
  Then  $\left(U\left(t\right)\right)_{t\ge 0}$ is given by the 
  \emph{Trotter product formula} \eqref{eq:trotter} 
  with uniform convergence for $t$ in compact intervals.
\end{corollary}

This raises the question weather \textit{\ref{item:closure})} is enough to
guarantee the convergence of  the \emph{Trotter product formula}
\eqref{eq:trotter} to the generated semigroup
$\left(U\left(t\right)\right)_{t\ge 0}$. 

In 2000 F. K\"{u}hnemund and M. Wacker \cite{KuhnWack} answered
negatively to this question providing a counterexample. However, their
counterexample is quite elaborated and makes use of results 
on strongly continuous evolution families.

Aim of this note is to show that very simple and natural
counterexamples come from linear hyperbolic systems of
partial differential equations. In particular we will show
that there exist very simple
examples of semigroups such that the sum of their generators
generates itself (without need of taking any closure) a semigroup,
but this semigroup can not be obtained through the
\emph{Trotter product formula} \eqref{eq:trotter}. 

A counterexample showing that \textit{\ref{item:stability})} is not
sufficient for the \emph{Trotter formula} to hold is already present
in the original paper by Trotter \cite{Trotter}. A
counterexample relevant for \emph{nonlinear semigroups} is given in
\cite{Kurz}.

In \cite{Corli,DynSysGuerra} there are extensions of Corollary
\ref{cor:12} to nonlinear semigroups in metric spaces requiring
commutator conditions instead of \textit{\ref{item:closure})}.
A commutator condition
was also used in a linear setting in \cite{MR1860489}.

\section{The counterexample}

Let $X$ be the Banach space of all vector valued bounded and
uniformly continuous functions on $\reali$,
$X=\mathrm{C}_{\mathrm{ub}}\left(\reali,\reali^{n}\right)$,
provided with the supremum norm
\begin{displaymath}
  \|f\|_{X}=\sup_{x\in\reali}\|f(x)\|,
\end{displaymath}
where $\|\cdot\|$ denotes the standard Euclidean norm in $\reali^{n}$.

Take now a $n\times n$ hyperbolic matrix $\mathcal{A}$, that is a 
real diagonalizable matrix with real
eigenvalues, and denote by
$\lambda_{1}^{\mathcal{A}}\le \ldots \le\lambda_{n}^{\mathcal{A}}$,\ \ 
$r_{1}^{\mathcal{A}},\ldots,r_{n}^{\mathcal{A}}$ and
$l_{1}^{\mathcal{A}},\ldots,l_{n}^{\mathcal{A}}$ its
eigenvalues, right eigenvectors and left eigenvectors respectively,
normalized in such a way that $\|r_{i}^{\mathcal{A}}\| = 1$ for $i=1,\ldots,n$
and $\left\langle l_{j}^{\mathcal{A}},r_{i}^{\mathcal{A}} \right\rangle =
\delta_{i,j}$ for $i,j=1,\ldots,n$, where
the symbol $\left\langle \cdot,\cdot\right\rangle$ denotes the scalar
product
in $\reali^{n}$ and $\delta_{i,j}$ is the Kronecker delta.

Then, one can define the following strongly continuous semigroup
(it is a group actually) on $X$:
\begin{equation}
  \label{eq:semigroup}
  \left[S(t)f\right](x) = \sum_{i=1}^{n}\left\langle l_{i}^{\mathcal{A}},
  f(x + \lambda_{i}^{\mathcal{A}}t)\right\rangle
  r_{i}^{\mathcal{A}},\qquad t\ge 0,\; x\in\reali \qquad f\in X.
\end{equation}

Indeed we have the following proposition.

\begin{proposition}
  \label{prop:21}
  If the hyperbolic matrix $\mathcal{A}$ is
  invertible, then $\left(S(t)\right)_{t\ge 0}$ defined in
  \eqref{eq:semigroup} is a
  strongly continuous group  on the Banach space $X$ and its
  generator $\left(A,\mathcal{D}(A)\right)$ is given by
  \begin{equation}
    \label{eq:gendef}
    \begin{split}
      \mathcal{D}(A)&
      =\left\{f\in X :\; f'\in X\right\},\\
      Af &= \mathcal{A}f'\quad
      \text{ for all }f\in \mathcal{D}(A).
    \end{split}
  \end{equation}
  If $f\in \mathcal{D}(A)$ the function $u(t,x)=\left[S(t)f\right](x)$
  is the unique classical solution to the following Cauchy problem for
  a hyperbolic system of first order partial differential equations
  \begin{displaymath}
    \begin{cases}
      \frac{\partial}{\partial
        t}u(t,x)=\mathcal{A}\frac{\partial}{\partial x}u(t,x)\\
      u(0,x)=f(x).
    \end{cases}
  \end{displaymath}
\end{proposition}

\begin{proof}
  If $f\in X$, $S(t)f\in X$ for all $t\in\reali$ since any linear combination of
  uniformly continuous bounded functions is a
  uniformly continuous bounded function. The group property
  follows from the relation $\left\langle
    l_{i}^{\mathcal{A}},r_{j}^{\mathcal{A}}\right\rangle=\delta_{ij}$
   by direct computations.
  Finally the uniform continuity of $f\in X$ implies that
  $\lim_{t\to 0}S(t)f=f$ in $X$. This concludes the proof that
  $S(t)$ is a strongly continuous group in $X$.

  Denote now by $\left(\tilde A,\mathcal{D}(\tilde A)\right)$
  its generator. We have to show that $\tilde A = A$ as defined in
  \eqref{eq:gendef}. Take first $f\in \mathcal{D}(A)$ and note that
  $f=\sum_{i=1}^{n}\left\langle l_{i}^{\mathcal{A}},
  f\right\rangle  r_{i}^{\mathcal{A}}$ and that
  the action of the hyperbolic matrix $\mathcal{A}$ can be expressed by
  $\mathcal{A}f'=\sum_{i=1}^{n}\left\langle l_{i}^{\mathcal{A}},
  f'\right\rangle \lambda_{i}^{\mathcal{A}}r_{i}^{\mathcal{A}}$.
  We can compute
\begin{displaymath}
 \begin{split} \lim_{t\to
0}&{\left\|\dfrac{S(t)f-f}{t}-\mathcal{A}f'\right\|_{X}}= \lim_{t\to
0}{\sup_x{\left\|\dfrac{\left[S(t)f\right](x)
-f(x)}{t}-\mathcal{A}f'(x)\right\|}}\\ &=\lim_{t\to
0}{\sup_x{\left\|\frac{1}{t}\sum_{i=1}^n{\left\langle l_i^{\mathcal{A}},
f(x+\lambda_i^{\mathcal{A}}t)-f(x)\right\rangle r_i^{\mathcal{A}}}
-\sum_{i=1}^n{\left\langle l_i^{\mathcal{A}},f'(x)\right\rangle
\lambda_{i}^{\mathcal{A}}r_{i}^{\mathcal{A}} }\right\|}}\\
&=\lim_{t\to 0}{\sup_x{\left\|\sum_{i=1}^n{\left\langle l_i^{\mathcal{A}},
\dfrac{f(x+\lambda_i^{\mathcal{A}}t)-f(x)}{t}-\lambda_{i}^{\mathcal{A}}f'(x)\right\rangle
r_i^{\mathcal{A}}}\right\|}}\\ &=\lim_{t\to
0}{\sup_x{\left\|\sum_{i=1}^n{\left\langle l_i^{\mathcal{A}},
      \lambda_{i}^{\mathcal{A}}\int_0^1\left({f'(x+\lambda_i^{\mathcal{A}}st)-f'(x)}\right)
      \;ds\right\rangle
r_i^{\mathcal{A}}}\right\|}}=0,
 \end{split}
\end{displaymath}
where we have used the uniform continuity of $f'$.
Hence if $f\in\mathcal{D}(A)$ then $f\in \mathcal{D}(\tilde A)$ and
$\tilde A f=Af$, that is $A \subset \tilde A$.

Now choose $f\in \mathcal{D}(\tilde A)$, then 
$\lim_{t\to 0}\frac{S(t)f-f}{t}=g$ in $X$ for some $g=\tilde A f\in X$.
This implies the following pointwise limit 
\begin{displaymath}
  \lim_{t\to 0}\sum_{i=1}^{n}\frac{\left\langle l_i^{\mathcal{A}},
    f\left(x+\lambda_i^{\mathcal{A}}t\right)
  \right\rangle - \left\langle l_i^{\mathcal{A}},
    f(x)  \right\rangle}{t}r_i^{\mathcal{A}}=\sum_{i=1}^{n}\left\langle l_i^{\mathcal{A}},
    \left[\tilde A f\right](x)
  \right\rangle r_i^{\mathcal{A}}.
\end{displaymath}
Since $\lambda_i^{\mathcal{A}}\neq 0$ for $i=1,...,n$ ($\mathcal{A}$
is invertible), we obtain that
$\left\langle l_i^{\mathcal{A}},f(x)\right\rangle$ is differentiable
with derivative
$\dfrac{1}{\lambda_i^{\mathcal{A}}}\left\langle l_i^{\mathcal{A}},
  \left[\tilde A f\right](x)\right\rangle$.
Then $f(x)=\sum_{i=1}^n{\left\langle l_i^{\mathcal{A}},
    f(x)\right\rangle r_i^{\mathcal{A}}}$ is differentiable with derivative
 $f'(x)=\sum_{i=1}^n{\dfrac{1}{\lambda_i^{\mathcal{A}}}\left\langle
     l_i^{\mathcal{A}},\left[\tilde A f\right](x)\right\rangle
   r_i^{\mathcal{A}}}\in X$, 
 which implies $f\in\mathcal{D}(A)$, $Af=\mathcal{A}f'=\tilde A f$, $\tilde A\subset
 A$
 concluding the proof that $A$ is the generator of $S(t)$.

 The last statement of the theorem follows by direct computations (see
 for instance \cite{bressanbook}). 

\end{proof}

Take now $\mathcal{A}$, $\mathcal{B}$ and $\mathcal{C}$
three $n\times n$ matrices
such that:
\begin{description}\item[\textbf{(H)}]
    they are invertible hyperbolic matrices, $\mathcal{C}=
    \mathcal{A}+\mathcal{B}$
 and
  the greatest eigenvalue of the matrix $\mathcal{C}$ is bigger than
  the sum of the greatest eigenvalues of the matrices $\mathcal{A}$,
  $\mathcal{B}$: $\lambda_{n}^{\mathcal{C}}>\lambda_{n}^{\mathcal{A}}
  +\lambda_{n}^{\mathcal{B}}$.
\end{description}
  \begin{remark}
  Given two arbitrary invertible hyperbolic matrices, their sum needs
  not to be even hyperbolic, but three matrices satisfying the
  property above exist, take for example
  \begin{displaymath}
    \mathcal{A}=
    \begin{pmatrix}
      0 & 1\\
      1 & 0
    \end{pmatrix},\quad \mathcal{B}=
    \begin{pmatrix}
      0 & 1\\
      4 & 0
    \end{pmatrix},\quad \mathcal{C}=
    \begin{pmatrix}
      0 & 2\\
      5 & 0
    \end{pmatrix}.
  \end{displaymath}
  For these matrices we have $\mathcal{C}=\mathcal{A}+\mathcal{B}$,
  $\lambda_{1}^{\mathcal{A}}=-1$, $\lambda_{2}^{\mathcal{A}}=1$;
  $\lambda_{1}^{\mathcal{B}}=-2$, $\lambda_{2}^{\mathcal{B}}=2$;
  $\lambda_{1}^{\mathcal{C}}=-\sqrt{10}$,
  $\lambda_{2}^{\mathcal{C}}=\sqrt{10}$; so that
  $\lambda_{2}^{\mathcal{C}}=\sqrt{10}> 3 = 1 + 2=
  \lambda_{2}^{\mathcal{A}}+\lambda_{2}^{\mathcal{B}}$.
\end{remark}

Given three matrices $\mathcal{A}$, $\mathcal{B}$ and $\mathcal{C}$
satisfying \textbf{(H)},
define now the following three strongly continuous semigroups on $X$:
\begin{equation}
  \label{eq:defsem}
  \begin{split}
  \left[S(t)f\right](x) &= \sum_{i=1}^{n}\left\langle l_{i}^{\mathcal{A}},
  f(x + \lambda_{i}^{\mathcal{A}}t)\right\rangle
  r_{i}^{\mathcal{A}},\qquad t\ge 0,\; x\in\reali \qquad f\in X,\\
  \left[T(t)f\right](x) &= \sum_{i=1}^{n}\left\langle l_{i}^{\mathcal{B}},
  f(x + \lambda_{i}^{\mathcal{B}}t)\right\rangle
  r_{i}^{\mathcal{B}},\qquad t\ge 0,\; x\in\reali \qquad f\in X,\\
  \left[U(t)f\right](x) &= \sum_{i=1}^{n}\left\langle l_{i}^{\mathcal{C}},
  f(x + \lambda_{i}^{\mathcal{C}}t)\right\rangle
  r_{i}^{\mathcal{C}},\qquad t\ge 0,\; x\in\reali \qquad f\in X.
  \end{split}
\end{equation}
Observe that, by Proposition \ref{prop:21},
their generators $\left(A,\mathcal{D}(A)\right)$,
$\left(B,\mathcal{D}(B)\right)$ and $\left(C,\mathcal{D}(C)\right)$
satisfy
\begin{displaymath}
  \mathcal{D}(A)=\mathcal{D}(B)=\mathcal{D}(C) =\mathcal{D}
  \dot=\left\{f\in X :\; f'\in X\right\},
\end{displaymath}
so that the sum of the generators $A$ and $B$,
\begin{displaymath}
  \left(A+B\right)f=Af+Bf=\mathcal{A}f'+\mathcal{B}f'
  =\left(\mathcal{A}+\mathcal{B}\right)f'=\mathcal{C}f'=Cf,
  \text{ for any }f\in\mathcal{D},
\end{displaymath}
is already closed and generates $\left(U(t)\right)_{t\ge 0}$.

We can now state our main result.

\begin{theorem}
  If the matrices $\mathcal{A},\;\mathcal{B},\;\mathcal{C}$
  satisfy \textbf{(H)}, then the semigroup $\left(U(t)\right)_{t\ge 0}$ is not given by the
  Trotter
  product formula of $\left(S(t)\right)_{t\ge
    0}$ and $\left(T(t)\right)_{t\ge 0}$, i.e. there exists a function
  $f\in X$ such that
  \begin{equation}
    \label{eq:trotter2}
    U(t)f \not= \lim_{n\to +\infty}\left(S\left(\frac{t}{n}\right)
    T\left(\frac{t}{n}\right)\right)^{n}f.
  \end{equation}
\end{theorem}

\begin{proof}
  Take the eigenvector $r_{n}^{\mathcal{C}}$ corresponding to the
  greatest eigenvalue $\lambda_{n}^{\mathcal{C}}$ of $\mathcal{C}$.
  Take now a function
  $\varphi\in\mathcal{C}^{\infty}\left(\reali,\reali\right)$
  satisfying $\varphi(x)>0$ for $x\in (0,1)$ and $\varphi(x)=0$
  for $x\not\in (0,1)$. Our function, which will be shown to satisfy
  \eqref{eq:trotter2},
  is given by $f=\varphi r_{n}^{\mathcal{C}}$. Indeed we can compute
  \begin{displaymath}
    \left[U(t)f\right](x)= \sum_{i=1}^{n}\left\langle l_{i}^{\mathcal{C}},
  \varphi(x + \lambda_{i}^{\mathcal{C}}t)r_{n}^{\mathcal{C}}\right\rangle
  r_{i}^{\mathcal{C}}=\varphi(x + \lambda_{n}^{\mathcal{C}}t)r_{n}^{\mathcal{C}},
\end{displaymath}
therefore $\left(U(t)f\right)(x)\not=0$ for
$x\in \left(-
  \lambda_{n}^{\mathcal{C}}t,1-\lambda_{n}^{\mathcal{C}}t\right)$.
Observe now that if $g\in X$ is such that $g(x)=0$ for all $x\le a$
for
some $a\in \reali$, then
\begin{equation}
  \label{eq:26}
 \left[S(t)g\right](x) = \sum_{i=1}^{n}\left\langle l_{i}^{\mathcal{A}},
  g(x + \lambda_{i}^{\mathcal{A}}t)\right\rangle
  r_{i}^{\mathcal{A}}=0\text{ for all }x\le a-\lambda_{n}^{\mathcal{A}}t  
\end{equation}
since when $x\le a - \lambda_{n}^{\mathcal{A}}t$, then
$x + \lambda_{i}^{\mathcal{A}}t\le a - 
  \left(\lambda_{n}^{\mathcal{A}}-
  \lambda_{i}^{\mathcal{A}}\right)t\le a$ and hence $g(x +
\lambda_{i}^{\mathcal{A}}t)=0$ for $i=1,\ldots,n$.
Analogously we have
\begin{equation}
  \label{eq:27}
 \left[T(t)g\right](x) = \sum_{i=1}^{n}\left\langle l_{i}^{\mathcal{B}},
  g(x + \lambda_{i}^{\mathcal{B}}t)\right\rangle
  r_{i}^{\mathcal{B}}=0\text{ for all }x\le a-\lambda_{n}^{\mathcal{B}}t.  
\end{equation}
Putting \eqref{eq:26} and \eqref{eq:27} together we obtain
\begin{equation}
  \label{eq:28}
  \left[S(t)T(t)g\right](x)=0\text{ for all }x \le a -
  \left(\lambda_{n}^{\mathcal{A}}+
    \lambda_{n}^{\mathcal{B}}\right)t.
\end{equation}
Using \eqref{eq:28} in the Trotter product one gets
\begin{displaymath}
\left[\left(S\left(\frac{t}{n}\right)
  T\left(\frac{t}{n}\right)\right)^{n}g\right](x)=0
\end{displaymath}
for all
\begin{displaymath}
x \le a - \left[\left(\lambda_{n}^{\mathcal{A}}+
    \lambda_{n}^{\mathcal{B}}\right)\frac{t}{n}
+\cdots +\left(\lambda_{n}^{\mathcal{A}}+
  \lambda_{n}^{\mathcal{B}}\right)\frac{t}{n}\right] =
a-\left(\lambda_{n}^{\mathcal{A}}+
    \lambda_{n}^{\mathcal{B}}\right)t
\end{displaymath}
Applying this to $f$ (which vanishes for $x\le 0$) we get
\begin{displaymath}
\left[\left(S\left(\frac{t}{n}\right)
  T\left(\frac{t}{n}\right)\right)^{n}f\right](x)=0\text{ for all }
x\le -\left(\lambda_{n}^{\mathcal{A}}+
    \lambda_{n}^{\mathcal{B}}\right)t,
\end{displaymath}
therefore
\begin{displaymath}
\lim_{n\to\infty}\left[\left(S\left(\frac{t}{n}\right)
  T\left(\frac{t}{n}\right)\right)^{n}f\right](x)=0\text{ for all }
x\le -\left(\lambda_{n}^{\mathcal{A}}+
    \lambda_{n}^{\mathcal{B}}\right)t.
\end{displaymath}
But since $\lambda_{n}^{\mathcal{C}}>\lambda_{n}^{\mathcal{A}}+
\lambda_{n}^{\mathcal{B}}$, we define $\xi(t)=
\min\left\{1-\lambda_{n}^{\mathcal{C}}t,-\left(\lambda_{n}^{\mathcal{A}}+
    \lambda_{n}^{\mathcal{B}}\right)t\right\}$, such that
for $x$ in the non empty (for $t>0$) interval
$\left(-\lambda_{n}^{\mathcal{C}}t,\xi(t)\right)$
\begin{displaymath}
\lim_{n\to\infty}\left[\left(S\left(\frac{t}{n}\right)
  T\left(\frac{t}{n}\right)\right)^{n}f\right](x)=0,
\end{displaymath}
while $\left[U(t)f\right](x)\not=0$ proving the Theorem.
\end{proof}

\begin{remark}
  This example shows that, in general, the solutions to the first order linear
  hyperbolic system
  \begin{displaymath}
      \frac{\partial}{\partial t}u =
      \left(\mathcal{A}+\mathcal{B}\right)\frac{\partial}{\partial x}u
  \end{displaymath}
  can not be obtained through the operator splitting composition
  of solutions to the two systems
  \begin{displaymath}
    \frac{\partial}{\partial t}u =
    \mathcal{A}\frac{\partial}{\partial x}u,\qquad
    \frac{\partial}{\partial t}u =
      \mathcal{B}\frac{\partial}{\partial x}u.
  \end{displaymath}
\end{remark}

\bibliographystyle{abbrv}
\bibliography{Main}

\end{document}